\documentclass{amsart}
\usepackage{marktext}
\usepackage{gimac}

\begin{document}
\title[Scheduling fixed length quarantines]{Scheduling fixed length quarantines to minimize the total number of 
fatalities during an epidemic}

\author[Feng]{Yuanyuan Feng}
\address{%
  Department of Mathematics,
  Pennsylvania State University,
  State College, PA 16802, USA.}
\email{yzf58@psu.edu}

\author[Iyer]{Gautam Iyer}
\address{%
  Department of Mathematical Sciences,
  Carnegie Mellon University,
  Pittsburgh, PA 15213, USA.}
\email{gautam@math.cmu.edu}

\author[Li]{Lei Li}
\address{%
  School of Mathematical Sciences, Institute of Natural Sciences, MOE-LSC,  Shanghai Jiao Tong University, Shanghai 200240,
  P.R. China.
}
\email{leili2010@sjtu.edu.cn}
\thanks{%
  The work of GI was partially supported by the National Science Foundation, under grant DMS-1814147, and the Center for Nonlinear Analysis.
  The work of L. Li was partially sponsored by NSFC 11901389 and Shanghai Sailing Program 19YF1421300. 
}

\begin{abstract}
  We consider a susceptible, infected, removed (SIR) system where the transmission rate may be temporarily reduced for a fixed amount of time.
  We show that in order to minimize the total number of fatalities, the transmission rate should be reduced on a single contiguous time interval, and we characterize this interval via an integral condition.
  We conclude with a few numerical simulations showing the actual reduction obtained.
\end{abstract}
\maketitle

\section{Introduction}

The SIR model was introduced by R. Ross and W. Hammer to model the spread of infectious diseases (see~\cites{KermackMcKendrickEA27,BrauerCastilloChavez12,Weiss13}).
In this model, we let $S$ denote the fraction of individuals that are \emph{susceptible} to the disease, $I$ the fraction of individuals that are \emph{infectious}, and $R$ the fraction of individuals that are \emph{removed}.
Removed individuals are those who have contracted the disease and have either recovered and acquired immunity, or have died.
The evolution of these three quantities is modelled by
\begin{subequations}
\begin{align}
  \label{e:S}
    \partial_t S &= -\beta SI \,,\\
  \label{e:I}
    \partial_t I &= \beta S I-\gamma I \,,\\
  \label{e:R}
    \partial_t R &= \gamma I \,.
\end{align}
\end{subequations}
Here $\beta$ is rate at which infectious individuals transmit the disease to the susceptible population, and $\gamma$ is the rate at which infectious individuals recover.

Typically $\beta$ and $\gamma$ are assumed to be model constants.
However, there are situations where one may be able to temporarily alter these constants.
One example of this is the current COVID-19 outbreak.
Here non-pharmaceutical interventions such as quarantines and social distancing were employed to temporarily reduce the transmission rate (see for instance~\cite{FergusonLaydonEA20,Rampini20,MaierBrockmann20,LaaroussiRachik20}).

In order to study this scenario, we assume that the transmission rate~$\beta$ is piecewise constant, and can take on one of two values: the normal transmission rate, $\beta_n$, and a reduced transmission rate, $\beta_q < \beta_n$, when quarantines / social distancing measures are in effect.
While these measures greatly reduce and may even completely stop the spread of the outbreak, for societal reasons one may not be able to impose them for extended periods of time.
This leads to a natural and interesting mathematical question:
\begin{quote}\itshape
  Given a fixed limit $T$ on the length of time social distancing / quarantines may be imposed, how should they be scheduled in order to minimize the total number of fatalities?
  Should the social distancing / quarantines be imposed in one contiguous interval, or broken up into multiple intervals?
  Should it be imposed early, when very few individuals are infected, or later when the infection levels are higher?
\end{quote}

To study this mathematically, we assume that a constant fraction of individuals who contract the disease will die.%
\footnote{%
  While this assumption is used in many situations, it does not always apply.
  For instance, during the COVID-19 pandemic the fatality rate was roughly constant when the number of infected individuals was small.
  However, when this number increased beyond the health-care capacity, the fatality rate almost doubled.}
In this case, minimizing the total number of fatalities is equivalent to minimizing $R(\infty) = \lim_{t \to \infty} R(t)$. 
Consequently, we will formulate all our results directly in terms of~$R(\infty)$. We remark that $R(\infty)=1-S(\infty)$ without social distancing / quarantines can be computed by the conservation of $I+S-\frac{\gamma}{\beta_n}\log S$ (see the Proof of Lemma~\ref{l:intCond}).

Formally, in equations~\eqref{e:S}--\eqref{e:R} the set of times when social distancing / quarantines are in effect may be an arbitrary measurable set.
However, it is only practical to impose and lift quarantines finitely many times, and thus we restrict our attention to this situation.
The main result of this paper shows that in order to minimize $R(\infty)$, it is always better to impose social distancing / quarantines in a single contiguous window of time, as opposed to splitting it up into multiple intervals (of the same total length).
Moreover, the best time window to impose social distancing / quarantines is often close to the time when the infection peaks, and we characterize this time window  analytically.
This is stated precisely below.

\begin{theorem}\label{t:main}
  Fix $T > 0$ and let $\mathcal T$ be the collection of all sets $\tau \subseteq [0, \infty)$ such that $\tau$ is a finite union of intervals with total length~$T$.
  Given $\tau \in \mathcal T$ define $\beta^\tau \colon [0, \infty) \to \R$ by
  \begin{equation*}
    \beta^\tau(t) =
      \begin{cases}
	\beta_q	& t \in \tau\,,\\
	\beta_n	& t \not\in \tau \,,
      \end{cases}
  \end{equation*}
  for some constants $0 < \beta_q < \beta_n$, and $\gamma > 0$.
  Let $S^\tau$, $I^\tau$, $R^\tau$ be the solution to
  \begin{equation}\label{e:SIRtau}
      \partial_t S^\tau = -\beta^\tau S^\tau I^\tau\,,
      \qquad
      \partial_t I^\tau = \beta^\tau S^\tau I^\tau-\gamma I^\tau \,,
      \qquad
      \partial_t R^\tau = \gamma I^\tau \,,
  \end{equation}
  with fixed initial data $I^\tau(0) = I_0 \in (0, 1)$, $S^\tau(0) = 1 - I_0$, $R^\tau(0) = 0$.
  Then, the set of times $\tau \in \mathcal T$ that minimizes $R^\tau(\infty)$ is always a single continuous interval of length $T$, and at least one of the following hold:
  \begin{enumerate}
    \item The minimizing interval $\tau$ is $[0, T]$.
    \item
      The minimizing interval~$\tau$ is characterized by the integral condition
      \begin{equation}\label{e:minCond}
	\int_\tau \frac{\gamma - \beta_n S^\tau}{I^\tau } \, dt = 0\,.
      \end{equation}
  \end{enumerate}
  If $\beta_n\leq \gamma$, then the first case above always holds.
  If instead $\beta_n > \gamma$, then there exists $\epsilon_0 > 0$ such that  the second case above holds for all $I_0 \in (0, \epsilon_0)$.
\end{theorem}
\begin{remark*}
  From the proof we will see that the $\epsilon_0$ above can be estimated by
  \begin{equation*}
    \epsilon_0 \approx \frac{1}{
      \beta_q T \max\set{1, e^{(\beta_q - \gamma)T }}}
      \paren[\Big]{1 - \frac{\gamma}{\beta_n} }.
  \end{equation*}
\end{remark*}

Recall that the \emph{basic reproduction number}, denoted by $\mathcal R^n_0$, is defined to be the ratio $\beta_n / \gamma$.
When $\mathcal R^n_0 \leq 1$ the transmission rate is slower than the recovery rate,  and the infection doesn't spread.
In this case the fraction of the population that is infected decreases monotonically.
Theorem~\ref{t:main} states that the total number of infected people is minimized if social distancing / quarantines are imposed at time $t = 0$, and this is not unexpected.

The more interesting case above is when $\mathcal R^n_0 > 1$.
In this case $\beta_n > \gamma$, and the infection will spread through the population.
One might now wonder whether it is more advantageous to impose social distancing / quarantines early when very few people are infected, or if its better to wait until a larger fraction of the population is infected, or if one should split up the quarantine into many short intervals.
Theorem~\ref{t:main} guarantees that then the most effective fixed length quarantine is a always a single contiguous time interval.
Moreover, when the second assertion of Theorem~\ref{t:main} holds, this interval contains the time when the infection peaks.
To see this, 
note that equation~\eqref{e:minCond} and the fact that $S$ is decreasing implies that $\beta_n S^\tau - \gamma$ is positive at the start of $\tau$, and negative at the end of $\tau$.
Thus, from~\eqref{e:I} we see that the disease is spreading at the start of $\tau$, attains its peak sometime during the time interval~$\tau$, and is dying out at the end of~$\tau$.
Hence the time interval $\tau$ that minimizes $R^\tau(\infty)$ must include the point when the number of infected individuals attains its peak. (See Figure~\ref{f:IR} for a simulation illustrating this.)


We also remark that when $\beta_n>\gamma$ and $I_0\geq \epsilon_0$, either conclusion~(1) or~(2) in Theorem~\ref{t:main} may hold, and we can not determine which one.
It is easy to see that if the population already has herd immunity (i.e.\ $I_0 \geq 1 - \gamma / \beta^n = 1 - 1/\mathcal R_0$), then the first conclusion in Theorem~\ref{t:main} must necessarily hold.
When $I_0 \in (\epsilon_0, 1 - 1/\mathcal R_0)$ then either conclusion~(1) or~(2) may hold, and we can not apriori determine which.

\subsection*{Discussion and further questions}
Before proceeding with the proof of Theorem~\ref{t:main}, we now provide a brief summary of related results and open questions that merit further study.

First we note that Theorem~\ref{t:main} can be reformulated more generally as an optimal control problem.
Namely, consider the case where adjusting the severity of the quarantine results in a variable transmission rate $\beta = \beta(t)$.
There is however a social and economic cost associated to imposing a quarantine measures, and this cost increases with the severity of the quarantine.
Of course, not imposing a quarantine results in more infected individuals and there is a social and economic cost associated with their care.
Combining these, we can quantify the total cost over the course of the infection as
\begin{equation*}
  \mathcal C(t) \defeq \int_0^\infty \paren[\big]{
      c_q\paren[\big]{ \beta_n - \beta(t) }
      + c_i( I(t) ) 
    } \, dt\,,
\end{equation*}
where $c_q$ and $c_i$ are increasing functions representing the costs associated to imposing quarantines, and the care of infected individuals respectively.

One can now study how the cost function~$\mathcal C$ can be minimized, subject to various practical constraints.
The constraint we study in this paper
requires~$\beta$ to be piecewise constant, only take on the values $\beta_q$ or $\beta_n$ and~$\int_0^\infty (\beta_n - \beta) \, dt = T (\beta_n - \beta_q)$.
Under this constraint, Theorem~\ref{t:main} finds the optimal~$\beta$ minimizing the cost function~$\mathcal C$ with $c_q = 0$ and $c_i(x) = x$.

Another constraint studied by Miclo et\ al.\ \cite{MicloWeibullEA20} is to only consider solutions for which $I(t) \leq \bar I_0$, for some exogenously specified level~$\bar I_0 \in (0, 1]$.
Here $\bar I_0$ represents the health care capacity, above which the mortality rate may be dramatically higher.
Under this constraint with the cost functions~$c_q(x) = x^+$ and $c_i(x) = 0$, Miclo et\ al.\ \cite{MicloWeibullEA20} show that 
the quarantine policy that minimizes~$\mathcal C$ is one where the infection grows unchecked until~$I = \bar I_0$, after which one imposes a quarantine and adjusts the severity to hold $I(t) = \bar I_0$ until herd immunity is achieved.
Recently, due to the COVID19 pandemic, many authors have studied various other costs and policies
both numerically and analytically, and we refer the reader to~\cites{Behncke00,AlvarezArgenteEA20,KisslerTedijantoEA20,KruseStrack20,Toda20}.
\smallskip

Another aspect that merits further study is a spatially-dependent system considering diffusion and population demography.
In this case the SIR system becomes a family of reaction diffusion systems~\cite{FitzgibbonLanglaisEA01,FitzgibbonLanglaisEA04,LaaroussiRachik20}.
In this setting one may naturally formulate an analog of Theorem~\ref{t:main} with the additional spatial component: given an upper bound on the product of the total time the quarantine is imposed and the size of the region it is imposed on, what is the optimal quarantine policy that minimizes the total number of fatalities?
This, however, is much harder to analyze and depends intrinsically on the spatial geometry, and we do not know if there will be a simple description of the optimal quarantine policy.
\smallskip

A third most important factor not considered in this paper is that of heterogeneous populations.
In a large group of humans there are various factors (such as social habits, or inherent tolerance) that contribute towards variance of the population.
One accounts for this by using a heterogeneous SIR model which divides the population into several homogeneous groups.
%
Counter-intuitively, in this case, a more severe quarantine can result in a \emph{higher} fraction of the population being infected (see~\cite{BrittonBallEA20}); an effect that is impossible to observe in a homogeneous population.


Numerous authors (see for instance~\cites{ChikinaPegden20,Rampini20,AcemogluChernozhukovEA20}) have also observed numerically that for heterogeneous populations quarantine measures that are targeted to each group are an order of magnitude more effective than un-targeted ones.
Theorem~\ref{t:main} can again be naturally formulated in this setting.
The proof, however, does not generalize, and we presently are unable to analytically characterize the optimal quarantine strategy in this case.
\smallskip

Finally, we mention  one novel feature that is unique to the recent COVID19 outbreak: asymptomatic carriers -- individuals who transmit the disease but show no symptoms.
Modeling their behavior is a newly developing, active area of study and we refer the reader to~\cite{MaierBrockmann20,ChenLuEA20,GanyaniKremerEA20}.
At present we do not know how best to model their behavior and how to reformulate Theorem~\ref{t:main} to capture their effect.

\subsection*{Plan of this paper}
In Section~\ref{s:thmproof} we state two lemmas required to prove Theorem~\ref{t:main}, and prove Theorem~\ref{t:main} modulo these lemmas.
In Section~\ref{s:lemmas} we prove both these lemmas.
Finally, in Section~\ref{s:numerics} we perform a few numerical simulations to illustrate Theorem~\ref{t:main}.

\section{Proof of Theorem~\ref{t:main}.}\label{s:thmproof}

Our aim in this section is to prove Theorem~\ref{t:main}.
The first step is to restrict our attention to social distancing / quarantines imposed on a contiguous interval, and show that the condition~\eqref{e:minCond} is necessary.
Fix $S_0 \in (0, 1)$, and set $I_0 = 1 - S_0$.
Given any $\tau \in \mathcal T$, we will subsequently denote $S^\tau, I^\tau, R^\tau$ to be the solution to~\eqref{e:SIRtau} with initial data $S^\tau(0) = S_0$, $I^\tau(0) =  I_0$, $R^\tau(0) = 0$.

Given any $S_0, I_0 \in (0, 1)$ with $S_0 + I_0 \leq 1$, define
\begin{equation}\label{eq:defQ}
  Q(S_0, I_0, T) = \int_0^T \frac{\gamma - \beta_n S_q(t)}{I_q(t)} \, dt \,,
\end{equation}
where $S_q, I_q$ solve~\eqref{e:S}--\eqref{e:I} with $\beta = \beta_q$ and initial data $S_q(0) = S_0$ and $I_q(0) = I_0$.
The necessity of~\eqref{e:minCond} for contiguous intervals $\tau$ can now be stated as follows.

\begin{lemma}\label{l:intCond}
  Let $t_0 \geq 0$ and $\tau = [t_0, t_0 + T]$.
  \begin{enumerate}
    \item
      Suppose $Q(S^\tau(t_0), I^\tau(t_0), T) > 0$ and $t_0 > 0$.
      Given $\delta \in (0, t_0)$, define $\sigma =  \sigma_\delta = [t_0 - \delta, t_0+T - \delta]$.
      Then, for all sufficiently small $\delta$, we must have $R^\sigma(\infty) < R^\tau(\infty)$.
      Moreover, if $S^{\tau}(t)=S^{\sigma}(t')$ for some $t>t_0+T$, $t'>t_0+T-\delta$ and sufficiently small~$\delta$, then we must have $I^{\sigma}(t')<I^{\tau}(t)$.

    \item
      On the other hand, suppose $Q(S^\tau(t_0), I^\tau(t_0)) < 0$.
      Now given any $\delta > 0$, define $\sigma = \sigma_\delta = [t_0 + \delta, t_0 + \delta + T]$.
      Then, for all sufficiently small $\delta$, we must have $R^\sigma(\infty) < R^\tau(\infty)$.
      Moreover, if $S^{\tau}(t)=S^{\sigma}(t')$ for some $t>t_0+T$ and $t'>t_0+T+\delta$ and sufficiently small~$\delta$, then we must have $I^{\sigma}(t')<I^{\tau}(t)$.
  \end{enumerate}
\end{lemma}

Next we show that $R^\tau(\infty)$ attains a minimum, and this minimum is attained when $\tau$ is a single contiguous interval.
Note that the set of all $\tau \in \mathcal T$ consisting of $m$ disjoint intervals can be identified with the set $\mathcal T_m \subseteq \R^{2m-1}$ defined by
\begin{equation}\label{e:tauTm}
  \mathcal T_m = \set[\Big]{ (t_1, \ell_1, \dots, t_{m-1}, \ell_{m-1}, t_m) \st 
    0 \leq t_i < t_i + \ell_i < t_{i+1},\
    \sum_{i=1}^{m-1} \ell_i < T
  }\,.
\end{equation}
Indeed, we identify the ordered tuple $(t_1, \ell_1, \dots, t_{m-1}, \ell_{m-1}, t_m)$ with the set $\tau \subseteq [0, \infty)$ defined by
\begin{equation*}
  \tau = \paren[\Big]{\bigcup_1^{m-1} [t_i, t_i + \ell_i]}
    \cup \brak[\Big]{ t_m, T - \sum_{j=1}^{m-1} \ell_j } \,.
\end{equation*}

Let $\bar{\mathcal T}_m$ denote the closure of $\mathcal T_m \subseteq \R^{2m -1}$, and define $\mathcal B_{m-1} = \bar{\mathcal T}_m - \mathcal T_m$.
Note that through the above identification, the set $\mathcal B_{m-1}$, represents a set of times $\tau \in \mathcal T$ with $m - 1$ (or less) disjoint intervals of total length $T$.
We will now show that even though $\bar{\mathcal T}_m$ is an unbounded set, the function $\tau \mapsto R^\infty(\tau)$ attains a minimum on $\bar{\mathcal T}_m$, and this minimum must be attained on $\mathcal B_{m-1}$.
\begin{lemma}\label{l:infTm}
  If $m > 1$, then the infimum of $R^\tau(\infty)$ over all $\tau \in \bar{\mathcal T}_m$ is attained at some point $\tau \in \mathcal B_{m-1}$.
\end{lemma}

Momentarily postponing the proofs of Lemmas~\ref{l:intCond} and~\ref{l:infTm}, we prove Theorem~\ref{t:main}.
\begin{proof}[Proof of Theorem~\ref{t:main}]
  Note that $\mathcal T$ can be viewed as an increasing union of the $\mathcal T_m$'s.
  By repeatedly applying Lemma~\ref{l:infTm}, we see that for any $m \geq 1$, the minimizer of $R^\tau(\infty)$ over all $\tau \in \mathcal T$ consisting of $m$ intervals or less must be attained when $\tau$ is a single contiguous interval.
  In this case, Lemma~\ref{l:intCond} forces the condition~\eqref{e:minCond} to be satisfied, unless $\tau = [0, T]$.
  This proves that either assertion~(1) or assertion~(2) in Theorem~\ref{t:main} must hold.

  For the last part of the theorem, suppose first $\beta_n \leq \gamma$.
  Since $S^\tau < 1$ and $I^\tau > 0$ this forces $Q(S^{\tau}(t_0), I^{\tau}(t_0), T)> 0$ for all $t_0\geq 0$.
  Thus condition~\eqref{e:minCond} can not be satisfied by any interval $\tau \in \mathcal T$, and hence the first assertion of Theorem~\ref{t:main} must hold.

  Finally, it only remains to show that when $\beta_n > \gamma$, there exists $\epsilon_0 > 0$ such that if $I(0) \in (0, \epsilon_0)$ then~\eqref{e:minCond} holds for the minimizing interval~$\tau$.
  Since we already know that one of the two conclusions~(1) or~(2) in Theorem~\ref{t:main} must hold, it suffices to show that the conclusion~(1) does not hold.
  To do this, by Lemma~\ref{l:intCond} it suffices to show that $Q(1 - \epsilon, \epsilon, T ) < 0$ for all $\epsilon \in (0, \epsilon_0)$.

  To see $Q(1 - \epsilon, \epsilon, T) < 0$, observe that~\eqref{e:I} implies
  \begin{equation*}
    I^{\tau}(t)=\epsilon \exp\paren[\Big]{
      \int_0^t(\beta_q S^\tau(t) -\gamma) \,ds}
    \leq \epsilon \max \set[\big]{ 1, e^{(\beta_q-\gamma) T} } \,,
  \end{equation*}
  for all $t \leq T$.
  Consequently,
  \begin{equation*}
    S^{\tau}(t)=(1-\epsilon)e^{-\beta_q\int_0^t I(s)ds}
      \geq (1-\epsilon)e^{-\beta_q \epsilon T \max \set{ 1, e^{(\beta_q-\gamma) T}}} \,,
  \end{equation*}
  for all $t \leq T$.
  Since $\gamma / \beta_n < 1$ by assumption, the above implies that $S^\tau(t) \geq \gamma / \beta_n$ for all $t \in \tau$, provided $\epsilon_0$ is sufficiently small.
  This forces~$Q(1 - \epsilon, \epsilon, T) < 0$, concluding the proof of Theorem~\ref{t:main}.
\end{proof}

\section{Proof of Lemmas}\label{s:lemmas}

This section is devoted to the proofs of Lemmas~\ref{l:intCond} and~\ref{l:infTm}.
We begin with Lemma~\ref{l:intCond}.

\begin{proof}[Proof of Lemma~\ref{l:intCond}]
Note that as $t \to \infty$, $I^\tau(t) \to 0$, and hence $R^\tau(\infty) = 1 - S^\tau(\infty)$.
  Since $S^\tau + I^\tau + R^\tau = 1$, minimizing $R^\tau(\infty)$ is the same as maximizing $S^\tau(\infty)$.
  In order to do this we study the behavior of $S^\tau$ as a function of $I^\tau$.
  Note first that when $\beta, \gamma$ are constants, solutions to~\eqref{e:S}--\eqref{e:I} conserve the quantity
  \begin{equation*}
    I + S - \frac{\gamma}{\beta} \log S \,.
  \end{equation*}
  This can readily be checked by differentiating and checking $\partial_t (I+S-\frac{\gamma}{\beta}\log S) = 0$.
  Thus, when no quarantine is imposed, one can compute $S(\infty)$ by solving the transcendental equation
  \begin{equation*}
    S(\infty) - \frac{\gamma}{\beta} \log S(\infty) = 1 - \frac{\gamma}{\beta} \log S_0\,.
  \end{equation*}

  In our case $\beta$ is not constant and there is no such explicit equation determining~$S^\tau(\infty)$.
  However, $\beta$ is piecewise constant, and so $I^\tau + S^\tau - \frac{\gamma}{\beta_n} \log S^\tau$ must be constant on every connected component of the complement of~$\tau$.
  Hence, we consider the family of curves $\mathcal C = \set{\Gamma_c \st c \in \R}$, where
  \begin{equation*}
    \Gamma_c \defeq \set[\big]{ (S, I) \in [0, 1]^2 \st
      S + I - \rho_n \log S = c\,,\ 
      S + I \leq 1 }\,,
    \quad\text{and}\quad
    \rho_n \defeq \frac{\gamma}{\beta_n}\,.
  \end{equation*}
  Note $\rho_n$ above is simply the reciprocal of the basic reproduction number $\mathcal R_0 = \beta_n / \gamma$.

  \begin{figure}[hbt]
    \includegraphics[width=.7\linewidth]{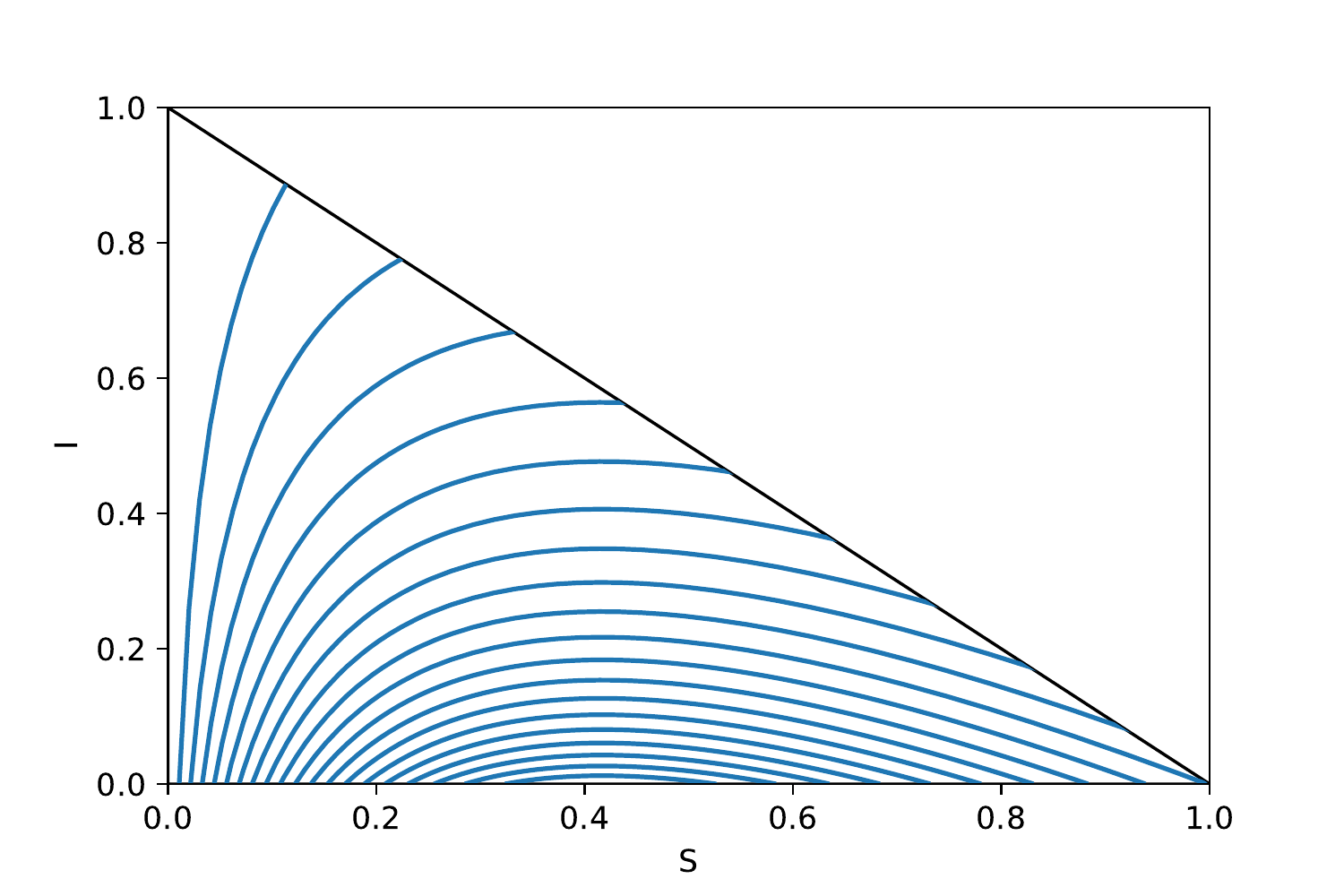}
    \caption{Various curves $\Gamma_c$ in the $S$-$I$ plane with $\mathcal R_0 = 2.4$. Only the portion of the curves that intersect the region $S \geq 0$, $I \geq 0$, $1 - S - I \geq 0$ are shown.}
    \label{f:GammaC}
  \end{figure}
  Each of the curves $\Gamma_c$ meet the line $I = 0$ at most twice (see Figure~\ref{f:GammaC}).
  The intersection when $S > \rho_n$ correspond to unstable equilibria, and so as $t \to \infty$, $(S^\tau(t), I^\tau(t))$ will approach some point $(S^\tau(\infty), 0)$ with $S^\tau(\infty) < \rho_n$.
  Thus, in order to maximize $S^\tau(\infty)$, we look for curves $\Gamma_c$ that meet the segment $\set{ I = 0,\ S \leq \rho_n}$ at an $S$-coordinate that is as large as possible.
  Implicitly differentiating $S - \rho_n \ln S = c$ we see that $\frac{dc}{dS} < 0$, and so smaller values of $c$ will lead to larger values of $S^\tau(\infty)$.
  \smallskip

  We will now prove the first assertion in Lemma~\ref{l:intCond}.
  The proof of the second assertion is similar.
  Choose $t_0 > 0$, assume $Q(S^\tau(t_0), I^\tau(t_0), T) > 0$ and let $\sigma = [t_0 - \delta, t_0 + T - \delta]$ for some small $\delta \in (0, t_0)$.
  For notational convenience, define
  \begin{align*}
    (S^\tau_0, I^\tau_0)
      &\defeq (S^{\tau}(t_0), I^{\tau}(t_0))\,,
      &
      (S^\tau_1,I^\tau_1) &\defeq (S^{\tau}(t_0+T), I^{\tau}(t_0+T))\,,
    \\
    (S^\sigma_0, I^\sigma_0)
      &\defeq (S^{\sigma}(t_0 - \delta), I^{\sigma}(t_0 -\delta))\,,
      &
      (S^\sigma_1,I^\sigma_1)
	&\defeq (S^{\sigma}(t_0+T-\delta), I^{\sigma}(t_0+T-\delta))\,,
  \end{align*}
  and let
  \begin{equation*}
    c_\tau \defeq S^\tau_1 + I^\tau_1 - \rho_n \ln S^\tau_1\,,
    \qquad
    c_\sigma \defeq S^\sigma_1 + I^\sigma_1 - \rho_n \ln S^\sigma_1\,.
  \end{equation*}
  We first claim
  \begin{equation}\label{e:ctildemc}
    c_\sigma - c_\tau
      = -\delta \beta_q (\rho_q - \rho_n) \, I^\tau_0 I^\tau_1 \, Q(S^\tau_0, I^\tau_0, T)
	+ O(\delta^2)\,.
  \end{equation}
  Once~\eqref{e:ctildemc} is established, our assumption on $Q$ implies $c^\sigma < c^\tau$.
  Using the argument in the previous paragraph, this in turn will imply $S^\sigma(\infty) > S^\tau(\infty)$ and hence $R^\sigma(\infty) < R^\tau(\infty)$ as desired.

  To prove~\eqref{e:ctildemc}, define the functions $g_n$ and $g_q$ by
  \begin{equation}\label{e:gdef}
    g_n(x) \defeq -x + \rho_n \log x
    \qquad\text{and}\qquad
    g_q(x) \defeq -x + \rho_q \log x\,.
  \end{equation}
  Using the fact that
  \begin{equation}
    \label{e:IeqgofS}
    I^\sigma(t) =  g_q(S^\sigma(t)) + I^\sigma_0 - g_q(S^\sigma_0)
    \quad\text{and}\quad
    I^\tau(t') =  g_q(S^\tau(t')) + I^\tau_0 - g_q(S^\tau_0) \,,
  \end{equation}
  for all $t \in \sigma$ and $t' \in \tau$,
  we note
  \begin{align}
    \nonumber
    c_\sigma - c_\tau
      &= I^\sigma_1 - g_n(S^\sigma_1)
	- \paren[\big]{ I^\tau_1 - g_n(S^\tau_1 ) }
    \\
    \nonumber
      &=
	I^\sigma_0 + (\rho_q - \rho_n) \log S^\sigma_1 - g_q( S^\sigma_0) 
	- \paren[\big]{
	  I^\tau_0 + (\rho_q - \rho_n) \log S^\tau_1 - g_q( S^\tau_0) 
	}
    \\
    \label{e:tildeCminusC1}
      &= (I^\sigma_0 - I^\tau_0)
	- (g_q(S^\sigma_0) - g_q(S^\tau_0) )
	+ (\rho_q - \rho_n) \paren[\big]{
	    \log S^\sigma_1 - \log S^\tau_1 } \,.
  \end{align}
  We now estimate each term on the right.

  The first two terms can be estimated quickly.
  Indeed equation~\eqref{e:SIRtau} shows
  \begin{equation}\label{e:perturbedstate}
    (S^{\sigma}_0,\ I^{\sigma}_0)
      =  (S^\tau_0,\ I^\tau_0) 
	+ \delta \paren[\Big]{
	\beta_n S^\tau_0 I^\tau_0,\ 
	\paren[\Big]{-1+ \frac{\rho_n}{S_0} }\beta_n S^\tau_0 I^\tau_0
      }
      +O(\delta^2)\,,
  \end{equation}
  and hence
  \begin{gather}
    \label{e:I0diff}
    I^\sigma_0 - I^\tau_0
      = \paren[\Big]{-1+ \frac{\rho_n}{S_0} }\beta_n S^\tau_0 I^\tau_0 \delta
	+ O(\delta^2)
    \\
    \label{e:gS0diff}
    g_q(S^\sigma_0) - g_q(S^\tau_0)
      = \paren[\Big]{-1 + \frac{\rho_q}{S^\tau_0}}
	\beta_n S^\tau_0 I^\tau_0
	\delta
	+ O(\delta^2)\,.
  \end{gather}

  The crux of the matter is the last term.
  For this, let $\Delta S = S^\sigma_1 - S^\tau_1$ and note that~\eqref{e:S} and~\eqref{e:IeqgofS} imply
  \begin{equation*}
    T = \int_{t_0}^{t_0 + T} \, dt
      = - \int_{t_0}^{t_0 + T} \frac{\partial_t S^\sigma}{\beta_q S^\sigma I^\sigma } \, dt
      = \int_{S^\sigma_1}^{S^\sigma_0}
	\frac{ds}{\beta_q s \paren{g_q(s) + I^\sigma_0 - g_q(S^\sigma_0)} }\,.
  \end{equation*}
  Using~\eqref{e:perturbedstate}--\eqref{e:gS0diff} and the above we see
  \begin{align*}
    \MoveEqLeft
    T = \int_{S^\sigma_1}^{S^\sigma_0}
      \frac{ds}{\beta_q s \paren{g_q(s) + I^\sigma_0 - g_q(S^\sigma_0)} }
    \\
    &= \int_{S^\tau_1 + \Delta S}^{S^\tau_0 + \delta \beta_n S^\tau_0 I^\tau_0}
      \frac{ds}{\beta_q s
	\paren{g_q(s) + I^\tau_0 - g_q(S^\tau_0)
	  - (\rho_q - \rho_n) \beta_n I^\tau_0 \delta
	} }
	+ O(\delta^2)
    \\
    &=
      \int_{S^\tau_1}^{S^\tau_0}
	\frac{ds}{\beta_q s
	  \paren{g_q(s) + I^\tau_0 - g_q(S^\tau_0) } }
      - \frac{\Delta S}{\beta_q S^\tau_1 I^\tau_1}
      + \frac{\delta \beta_n }{\beta_q}
    \\
      &\qquad
	+ \delta \int_{S^\tau_1}^{S^\tau_0}
	  \frac{(\rho_q - \rho_n) \beta_n I^\tau_0}{\beta_q s (g_q(s) + I^\tau_0 - g_q(S^\tau_0))^2} \, ds
	+ O(\delta^2)
  \end{align*}
  Using~\eqref{e:S} and~\eqref{e:IeqgofS} this simplifies to
  \begin{align*}
    T &= \int_{t_0}^{t_0 + T} \, dt
      - \frac{\Delta S}{\beta_q S^\tau_1 I^\tau_1}
      + \delta \paren[\Big]{
	  \frac{\beta_n }{\beta_q}
	  + (\rho_q - \rho_n) \beta_n I^\tau_0 \int_{t_0}^{t_0 + T} \frac{dt}{I^\tau(t)}
	}
      + O(\delta^2)\,,
  \end{align*}
  and hence
  \begin{equation}\label{e:DeltaS}
    \Delta S = \delta \beta_n S^\tau_1 I^\tau_1
      \paren[\Big]{
	1
	+ (\rho_q - \rho_n) \beta_q I^\tau_0
	  \int_{t_0}^{t_0 + T} \frac{dt}{I^\tau}
      }
      + O(\delta^2)\,.
  \end{equation}

  Now, using~\eqref{e:I0diff}, \eqref{e:gS0diff} and~\eqref{e:DeltaS} in~\eqref{e:ctildemc} we see
  \begin{align}
    \nonumber
    c_\sigma - c_\tau
      &= (\rho_q - \rho_n)\paren[\Big]{
	  \frac{\Delta S}{S^\tau_1}
	  - \beta_n I^\tau_0 \delta }
	+ O(\delta^2)
    \\
      \label{e:oneoverItau}
      &= \delta \beta_n (\rho_q - \rho_n) I^\tau_0 I^\tau_1
	\paren[\Big]{
	  \paren[\Big]{
	      \frac{1}{I^\tau_0} - \frac{1}{I^\tau_1}
	    }
	  + \beta_q (\rho_q - \rho_n) \int_{t_0}^{t_0 + T} \frac{dt}{I^\tau}
	}
	+ O(\delta^2)\,.
  \end{align}
  Since
  \begin{equation*}
    \frac{1}{I^\tau_0} - \frac{1}{I^\tau_1}
    = \int_{I^\tau_0}^{I^\tau_1} \frac{di}{i^2}
    = \int_{t_0}^{t_0 + T} \frac{\beta_q S^\tau - \gamma }{I^\tau } \, dt\,,
  \end{equation*}
  we see
  \begin{equation*}
    c_\sigma - c_\tau
      = \delta \beta_q ( \rho_q - \rho_n ) I^\tau_0 I^\tau_1
        \int_{t_0}^{t_0+T}
          \frac{
	    \beta_n S^\tau - \gamma
          }{I^\tau}
          \, dt
	+ O(\delta^2)\,,
  \end{equation*}
  proving~\eqref{e:ctildemc} as claimed.
  As explained earlier, this will prove~$R^\sigma(\infty) < R^\tau(\infty)$ as desired.

  It remains to show that if for some $t > t_0 + T$ and $t' > t_0 + T - \delta$ we have $S^\tau(t) = S^\sigma(t')$, then we must have $I^\sigma(t') < I^\tau(t)$.
  To see this, we consider the phase portrait the curve $I^\sigma$ vs $S^\sigma$ for times $t' \geq t_0 + T - \delta$, and phase portion of the curve $I^\tau$ vs $S^\tau$ for times $t \geq t_0 + T$.
  Since the times we consider are after the end of the intervals $\tau$ and $\sigma$, both these curves must be members of~$\mathcal C$.
  We already know $R^\sigma(\infty) < R^\tau(\infty)$, and hence $S^\sigma(\infty) > S^\tau(\infty)$.
  This means that in the $I$ vs $S$ plane, the curve parametrized by $(S^\sigma(t'), I^\sigma(t'))$ for $t' > t_0 + T - \delta$ must lie \emph{below} the curve parametrized by $(S^\tau(t), I^\tau(t))$ for $t > t_0 + T$.
  Thus if $S^\tau(t) = S^\sigma(t')$ for some $t > t_0 + T$, $t' > t_0 + T - \delta$, we must have $I^\sigma(t') < I^\tau(t)$.
  This finishes the proof.
\end{proof}

An immediate corollary to Lemma~\ref{l:intCond} is that if the minimizer~$\tau \in \mathcal T$ is not a contiguous interval, then the integral condition~\eqref{e:minCond} must be satisfied on the last contiguous interval in~$\tau$.

\begin{lemma}\label{l:lastintervalCond}
  Suppose $\tau = \bigcup_{i = 1}^m [t_i, t_i + \ell_i]$, with $0 < t_i < t_i + \ell_i < t_{i+1}$, and $\sum \ell_i = T$.
  Let $\tau' = \bigcup_{i = 1}^{m-1} [t_i, t_i + \ell_i]$, and $Q_m = Q( S^{\tau'}(t_m), I^{\tau'}(t_m), \ell_m )$.
  \begin{enumerate}
    \item
      If $Q_m > 0$ then there exists $\delta \in (0, t_m - t_{m-1} - \ell_{m-1})$ such that for
      \begin{equation*}
	\sigma = \tau' \cup [t_m - \delta, t_m - \delta + \ell_m]
      \end{equation*}
      we have $R^\sigma(\infty) < R^\tau(\infty)$.

    \item
      If $Q_m < 0$ then there exists $\delta > 0$ such that for
      \begin{equation*}
	\sigma = \tau' \cup [t_m + \delta, t_m + \delta + \ell_m]
      \end{equation*}
      we have $R^\sigma(\infty) < R^\tau(\infty)$.
  \end{enumerate}
\end{lemma}
\begin{proof}
  Applying Lemma~\ref{l:intCond} with $T = \ell_m$ with initial data~$S^\tau(t_{m-1} + \ell_{m-1})$, $I^\tau(t_{m-1} + \ell_{m-1})$ immediately yields Lemma~\ref{l:lastintervalCond}.
  (Note, while the convention $R^\tau(0) = 0$ was used throughout Section~\ref{s:thmproof}, it is not required for Lemma~\ref{l:intCond}, and was not used in the proof of Lemma~\ref{l:intCond}.
  Thus our application of Lemma~\ref{l:intCond} above is valid.)
\end{proof}

Our next result establishes an ``order preserving'' property of solutions to~\eqref{e:SIRtau}.
Fix $\delta > 0$ and $S_0, I_0 \in (0, 1)$ with $S_0 + I_0 \leq 1$.
Let $\tau = [0, T]$, and consider the following two solutions to~\eqref{e:SIRtau}.
The first, denoted by $S, I$, with initial data $(S_0, I_0)$, and the second, denoted by $(S^\delta, I^\delta)$ with initial data~$(S_0, I_0 - \delta)$.
In the $S$-$I$ plane, must the curve $(S^\delta, I^\delta)$ lie below that of~$(S, I)$?

One might, at first sight, think this is certainly true.
However, since $\beta^\tau$ depends on~$t$, the system~\eqref{e:SIRtau} is not autonomous, and so it is possible for the curves $(S^\delta, I^\delta)$ and $(S, I)$ to cross each other.
Various such non-monotonicity phenomena were studied in~\cite{ChikinaPegden20a}. 
We will also provide a simple example of this shortly.

Fortunately, it turns out that if additionally we assume $Q(S_0, I_0, T) = 0$, then $(S^\delta, I^\delta)$ must eventually lie below the curve $(S, I)$.
This is all we need in the proof, and is stated as our next lemma.

\begin{lemma}\label{l:orderpreserving}
  Let $S_0, I_0 \in (0, 1)$ with $S_0 + I_0 \leq 1$, and $\delta \in (0, I_0)$.
  Let $\tau = [0, T]$, $(S, I)$ solve~\eqref{e:SIRtau} with initial data $S(0) = S_0$, $I(0) = I_0$, and let $(S^\delta, I^\delta)$ solve~\eqref{e:SIRtau} with initial data~$S^\delta(0) = S_0$, $I^\delta(0) = I_0 - \delta$.
  If $Q(S_0, I_0, T) = 0$, then for all sufficiently small~$\delta$ we must have $R^\delta(\infty) < R(\infty)$.
  (Here $R = 1 - S - I$, and $R^\delta = 1 - S^\delta - I^\delta$.)
\end{lemma}
\begin{proof}
  Let $S_1 = S(T)$, $I_1 = I(T)$, $S^\delta_1 = S^\delta(T)$, and $I^\delta_1 = I^\delta(T)$.
  We will first show
  \begin{equation}\label{e:reversecondition1}
    I_1+S_1-\rho_n\log S_1 > I_1^{\delta}+S_1^{\delta}-\rho_n\log S_1^{\delta}
  \end{equation}
  if and only if
  \begin{gather}\label{e:reversecondition}
       \beta_q \paren{\rho_q - \rho_n}
      I(T)
      \int_{0}^{T}\frac{1}{I(t)} \, dt
      <1 \,.
  \end{gather}

  To see this, define
  \begin{equation*}
    c_1 = I_1 + S_1 - \rho_n \log S_1\,,
    \quad\text{and}\quad
    c^\delta_1 = I^\delta_1 + S^\delta_1 - \rho_n \log S^\delta_1
  \end{equation*}
  We claim
  \begin{equation}\label{e:cdeltaminusc}
    c^\delta_1 - c_1 = \delta \paren[\Big]{
	\beta_q (\rho_q - \rho_n) I_1 \int_0^T \frac{1}{I(t)} \, dt - 1
      }
      + O(\delta^2)\,,
  \end{equation}
  from which the equivalence of~\eqref{e:reversecondition1} and~\eqref{e:reversecondition} immediately follows.

  The proof of~\eqref{e:cdeltaminusc} is very similar to the proof of Lemma~\ref{l:intCond}.
  Let~$g_q$ be defined by~\eqref{e:gdef}, and let $c_0 = I_0 - g_q(S_0)$.
  Since $I - g_q(S)$ is conserved, we note
  \begin{equation*}
    I(t) = g_q(S(t)) + c_0\,,
    \qquad\text{and}\qquad
    I^\delta(t) = g_q(S^\delta(t) + c_0 - \delta\,,
  \end{equation*}
  for all $t \in [0, T]$.
  From~\eqref{e:S}, we see
  \begin{equation*}
    -\int_0^T \frac{\partial_t S}{S I}
    = \beta_q T
    = -\int_0^T \frac{\partial_t S^\delta}{S^\delta I^\delta}\,,
  \end{equation*}
  and hence
  \begin{equation*}
    \int_{S_1}^{S_0}
      \frac{ds}{s (g_q(s) + c_0)}
    = \int_{S^\delta_1}^{S_0}
	\frac{ds}{s (g_q(s) + c_0 - \delta)} \,.
  \end{equation*}
  Taylor expanding as in the proof of Lemma~\ref{l:intCond} immediately shows
  \begin{equation*}
    \Delta S \defeq S^\delta_1 - S_1
      =	\delta S_1 I_1 \int_{S_1}^{S_0} \frac{ds}{s (g_q(s) + c_0)^2 }
	+ O(\delta^2)
      =	\delta \beta_q S_1 I_1 \int_0^T \frac{dt}{I}
	+ O(\delta^2)\,.
  \end{equation*}
  Consequently,
  \begin{equation*}
    c^\delta_1 - c_1 =
      -\delta + (\rho_q - \rho_n) (\log S^\delta_1 - \log S_1)
      = -\delta \frac{\rho_q - \rho_n}{S_1} \Delta S + O(\delta^2)\,,
  \end{equation*}
  from which~\eqref{e:cdeltaminusc} follows.
  This establishes the equivalence of~\eqref{e:reversecondition1} and~\eqref{e:reversecondition}.

  Now we use this equivalence to prove Lemma~\ref{l:orderpreserving}.
  Using~\eqref{e:oneoverItau} we see that $Q(S_0, I_0, T) = 0$ is equivalent to
  \begin{equation}
    \frac{1}{I_0} - \frac{1}{I_1} + \beta_q (\rho_q-\rho_n) \int_{0}^{T}\frac{dt}{I} = 0\,.
  \end{equation}
  This implies 
  \begin{equation}
    \beta_q (\rho_q-\rho_n)I_1\int_{0}^{T}\frac{dt}{I} = 1-\frac{I_1}{I_0}<1
  \end{equation}
  as desired.
\end{proof}

\begin{remark*}
  Before proceeding further, we provide an example showing that Lemma~\ref{l:orderpreserving} is false if we drop the assumption that $Q(S_0, I_0, T) = 0$.
  To do this note that in the above proof we establish the equivalence between~\eqref{e:reversecondition1} and~\eqref{e:reversecondition} without using the assumption that $Q(S_0, I_0, T) = 0$.
  Thus, if we produce an example where~\eqref{e:reversecondition} is false, then~\eqref{e:reversecondition1} will also be false, which is what we want.

  To construct this example, suppose $\rho_n$ is very small, and $\rho_q < 1$.
  Choose $T$ such that $S(T) = \rho_q$, and let $\tau = [0, T]$.
  By making $I_0$ sufficiently small, $T$ can be made arbitrarily large.
  We choose $I_0$ large enough so that
  \begin{equation*}
    T > \frac{1}{\beta_q (\rho_q - \rho_n)}\,.
  \end{equation*}
  Now for $t \leq T$, note~\eqref{e:I} implies
  \begin{equation*}
    \partial_t I = \beta_q I (S - \rho_q) > 0\,.
  \end{equation*}
  Hence the left hand side in~\eqref{e:reversecondition} can be estimated by
  \begin{align*}
    \beta_q \paren{\rho_q - \rho_n}
	I(T)
	\int_{0}^{T}\frac{1}{I(t)} \, dt
      \geq \beta_q \paren{\rho_q - \rho_n} T
      > 1\,,
  \end{align*}
  by our choice of~$I_0$.
  This in turn implies~\eqref{e:reversecondition1} is false, and hence~$R^\delta(\infty) > R(\infty)$ for all sufficiently small~$\delta$, contrary to the conclusion of Lemma~\ref{l:orderpreserving}.
\end{remark*}

Next, to prove of Lemma~\ref{l:infTm}, we need a few elementary properties of~\eqref{e:SIRtau}.
\begin{lemma}\label{l:SIRbehavior}
  Given~$\tau \in \mathcal T$, let $(S^\tau, I^\tau)$ solve~\eqref{e:SIRtau} with initial data $I^\tau(0) = I_0 \in (0, 1)$ and $S^\tau(0) = 1 - I_0$.
  \begin{enumerate}
    \item
      For every $\tau \in \mathcal T$, the function $t \mapsto S^\tau(t)$ is strictly decreasing, and $I^\tau(\infty) = 0$.
    \item
      There exists $T_* = T_*(\beta_n, \gamma, T, I_0)$ such that for every $\tau \in \mathcal T$, we have
      \begin{equation*}
	0<S^{\tau}(t)< \frac{\gamma}{\beta_n}\,,
	\quad \text{for all } t>T_* \,.
      \end{equation*}
    \item
      For every $m \geq 1$, the functions $\tau \mapsto R^{\tau}(\infty)$ is continuous on $\bar{\mathcal T}_m$.
      Moreover, $R^\tau(\infty) = 1 - S^{\tau}(\infty) \in (0, 1)$.
  \end{enumerate}
\end{lemma}
\begin{proof}[Proof of Lemma \ref{l:SIRbehavior}]
  From~\eqref{e:SIRtau} we see that $S^\tau, I^\tau > 0$ for all  $t > 0$.
  This implies $\partial_t S^\tau < 0$, showing $S^\tau$ is a decreasing function.
  Since $\tau$ is always a bounded set, $(S^\tau, I^\tau)$ satisfy~\eqref{e:S}--\eqref{e:I} with constant~$\beta$ for all large time.
  In this case it is well known that~$I^\tau$ decreases exponentially to~$0$ (see for instance~\cite{Weiss13,BrauerCastilloChavez12}).
  \medskip

  For the second assertion, note that $S^\tau(t) < 1$ for all $t > 0$.
  Thus if $\beta_n \leq \gamma$ we are done.
  Now we suppose $\beta_n > \gamma$.
  In this case if $1 - I_0 < \gamma / \beta_n$, then we simply choose $T_* = 0$.
  If not, suppose for some $T_0 \geq 0$ we have $S^\tau(T_0) \geq \gamma / \beta_n$.
  Since $S^\tau$ is decreasing, this implies $S^\tau(t) \geq \gamma / \beta_n$ for all $t \leq T_0$.
  Using~\eqref{e:SIRtau} we see that this means
  \begin{equation*}
    \partial_t I^\tau \geq
      \begin{cases}
	0 & t \in [0, T_0] - \tau\,,\\
	-\gamma I^\tau & t \in [0, T_0] \cap \tau\,.
      \end{cases}
  \end{equation*}
  Since the total length of~$\tau$ is~$T$, this implies
  \begin{equation*}
    I^\tau(t) \geq I_0 e^{-\gamma T}
    \quad\text{for all } t \leq T_0\,.
  \end{equation*}
  Using this in~\eqref{e:SIRtau} shows that
  \begin{equation*}
    S^\tau(t) \leq (1 - I_0) \exp\paren[\Big]{ -t I_0 e^{-\gamma T} }
    \quad\text{for all } t \leq T_0\,.
  \end{equation*}
  Since by assumption~$S^\tau(T_0) \geq \gamma / \beta_n$, this implies
  \begin{equation*}
    T_0 \leq \frac{e^{\gamma T}}{I_0} \log\paren[\Big]{\frac{\beta_n (1 - I_0)}{\gamma} } \defeq T_*\,.
  \end{equation*}
  Since $T_*$ is independent of~$\tau$, we obtain the second assertion of the lemma.
  \medskip

  Finally it remains to prove that $\tau \mapsto R^\tau(\infty)$ is continuous on $\bar{\mathcal T}_m$.
  To fix notation, identify $\tau$ with a subset of $[0, \infty)$ using~\eqref{e:tauTm}.
  By standard ODE theory we know that the function $\tau \mapsto (S^\tau(t_m + \ell_m), I^\tau(t_m + \ell_m))$ is continuous.
  After time $t_m + \ell_m$, we note that~$(S^\tau, I^\tau)$ satisfy~\eqref{e:S}--\eqref{e:I} with $\beta = \beta_n$.
  In this case it is know that
  \begin{equation*}
    S^{\tau}(\infty)
      = S^{\tau}(t_m+\ell_m)
	\exp\paren[\Big]{
	  -\frac{\beta_n}{\gamma}
	    \brak[\Big]{ S^{\tau}(t_m+\ell_m)+I^{\tau}(t_m+\ell_m)-S^{\tau}(\infty)}}\,.
  \end{equation*}
  The implicit function theorem now shows $\tau \mapsto S^\tau(\infty)$ is continuous.
  Since~$I^\tau(\infty) = 0$, and $S^\tau + I^\tau + R^\tau = 1$, this in turn implies $\tau \mapsto S^\tau(\infty)$ is continuous on $\bar{\mathcal T}_m$.
\end{proof}

Finally, we need to rule out the possibility that the infimum of $R^\tau(\infty)$ over $\mathcal T_m$ is attained at~$\infty$.
This is our next Lemma.
\begin{lemma}\label{l:inc}
  Let $T_*$ be as in Lemma~\ref{l:SIRbehavior}, and let $\tau = \cup_{i=1}^{n}[t_i,t_i+\ell_i] \in \mathcal T$ for some $n\geq 1$.
  Fix $\ell > 0$.
  For any $t \geq \max\set{t_{n}+\ell_{n}, T_*}$, define $\sigma(t) = \tau \cup [t, t+\ell] \in \mathcal T$.
  The function $t \mapsto R^{\sigma(t)}(\infty)$ is increasing in $t$.
\end{lemma}
\begin{proof}
  Note that for $t > T_*$, we must have~$S^{\sigma(t)}(t) \leq \rho_n$.
  Hence, by~\eqref{eq:defQ} we must have $Q( S^{\sigma(t)}(t), R^{\sigma(t)}(t), \ell ) > 0$.
  Now by Lemma~\ref{l:intCond} part~(1), we see that $R^{\sigma(t - \delta)}(\infty) < R^{\sigma(t)}$ for all sufficiently small $\delta$, finishing the proof.
\end{proof}

With the above tools, we are now ready to prove Lemma~\ref{l:infTm}.
\begin{proof}[Proof of Lemma \ref{l:infTm}]
  Let $T_*$ be as in Lemma~\ref{l:SIRbehavior}.
 Fix any $T^* > T+T_*$.
  Define $\mathcal T_m^* \subseteq \mathcal T_m$ by
  \begin{equation*}
    \mathcal T_m^* \defeq
      \set[\Big]{ (t_1, \ell_1, \dots, t_{m-1}, \ell_{m-1}, t_m) \st 
	    0< t_i < t_i + \ell_i < t_{i+1},\
	    \sum_{i=1}^{m-1} \ell_i < T,\ 
	    t_m < T^*
	  }\,.
  \end{equation*}
 As before, we identify $\tau\in \mathcal T_m^*$ with $\tau=(\cup_{i=1}^{m-1}[t_i, t_i+\ell_i])\cup [t_m, t_m+T-\sum_{j=1}^{m-1}\ell_j]\in \mathcal{T}$.
 Let $\bar{\mathcal T}_m^*$ denote the closure of $\mathcal T_m$ in $\R^{2m-1}$.  
  Note that for any $\tau \in \bar{\mathcal T}_m$, if the last contiguous interval in~$\tau$ starts after time $T_*$, then Lemma~\ref{l:inc} implies that shifting this interval to the left decreases $R(\infty)$.
  Moreover, if more than one contiguous interval in~$\tau$ starts after $T_*$, then repeatedly applying Lemma~\ref{l:inc} shows that they can be merged and shifted left to decrease $R(\infty)$, and $t_m$ can be shifted to be smaller than $T^*$.
  This implies
  \begin{equation*}
    \inf_{\tau \in \bar{\mathcal T}_m} R^\tau(\infty)
      = \inf_{\tau \in \bar{\mathcal T}^*_m} R^\tau(\infty)\,.
  \end{equation*}
  Since $\tau \mapsto R^\tau(\infty)$ is continuous (Lemma~\ref{l:SIRbehavior}), and $\mathcal T_m^*$ is compact, the infimum must be attained.
  Hence, there exists $\tau = (t_1, \ell_1, \dots, t_m, \ell_m) \in \bar{\mathcal T}_m$ such that $R^\tau(\infty) = \inf_{\tau \in \bar{\mathcal T}_m} R^\tau(\infty)$.

  We now claim that when $m > 1$, we must have $\tau \in \mathcal B_{m-1}$.
  To prove this it suffices to show that $\tau \not\in \mathcal T_m$.
  Suppose, for sake of contradiction, that~$\tau \in \mathcal T_m$.
  Let $\tau'$ and $Q_m$ be as in Lemma~\ref{l:lastintervalCond}.
  Since $\tau$ minimizes $R^\tau(\infty)$ by assumption, Lemma~\ref{l:lastintervalCond} implies that $Q_m = 0$. Then, $S^{\tau}(t)>\rho_n$ for all $t\in [t_{m-1}, t_{m-1}+\ell_{m-1})$ so that $Q(S^{\tau}(t_{m-1}), I^{\tau}(t_{m-1}), \ell_{m-1})<0$. 
  Let $\delta > 0$ be small and define $\sigma'$ by
  \begin{equation*}
    \sigma' = \paren[\Big]{ \bigcup_{i=1}^{m-1} [t_i, t_i + \ell_i] }
      \cup [t_{m-1} + \delta, t_{m-1} + \delta + \ell_{m-1} ]\,.
  \end{equation*}
  By continuity of solutions, there must exist $t_m' > t_{m-1} + \delta + \ell_{m-1}$ such that $S^{\sigma'}(t_m') = S^{\tau'}(t_m)$ when $\delta$ is small enough.
  Define~$\sigma = \sigma' \cup [t_m', t_m' + \ell_m]$, and observe that by Lemma~\ref{l:lastintervalCond} we must have $I^\sigma(t_m') = I^{\sigma'}(t_m') < I^\tau(t_m)$.
  Now, since~$Q_m = 0$, Lemma~\ref{l:orderpreserving} implies that~$R^\sigma(\infty) < R^\tau(\infty)$ for small $\delta$, as the gap between $I^\sigma(t_m')$ and $I^\tau(t_m)$ tends to zero when $\delta\to 0$ by continuity.
  Thus we have produced $\sigma \in \mathcal T$ such that $R^\sigma(\infty) < R^\tau(\infty)$, contradicting our assumption.
  This finishes the proof.
\end{proof}

\section{Numerical simulations.}\label{s:numerics}

We conclude this paper with numerical simulations showing how significant the reduction in $R(\infty)$ is.
We will also fix the time window when social distancing / quarantines  are in effect to be $30$ days (i.e.\ $T = 30$).
Choose $\gamma = 1/14$, corresponding to a recovery time of $14$ days, and consider a disease for which $\mathcal R_0 = 2.1$ normally, and $\mathcal R_0 = 0.8$ when social distancing / quarantines are in effect.
Figure~\ref{f:IR} (left) shows how the fraction of infected and removed individuals evolves with time.
In this case we see that $R(\infty)$ reduces from $0.82$ when no quarantine is imposed to $0.70$ when a $30$ day contiguous quarantine is optimally imposed.
As expected, we see that the optimal quarantine starts a little before the (unquarantined) infection levels peak, and ends a little after it.
Since the population attains herd immunity exactly when the infection levels peak, the unquarantined population attains herd immunity sometime during the optimal quarantine.
\begin{figure}[hbt]
  \includegraphics[width=.48\linewidth]{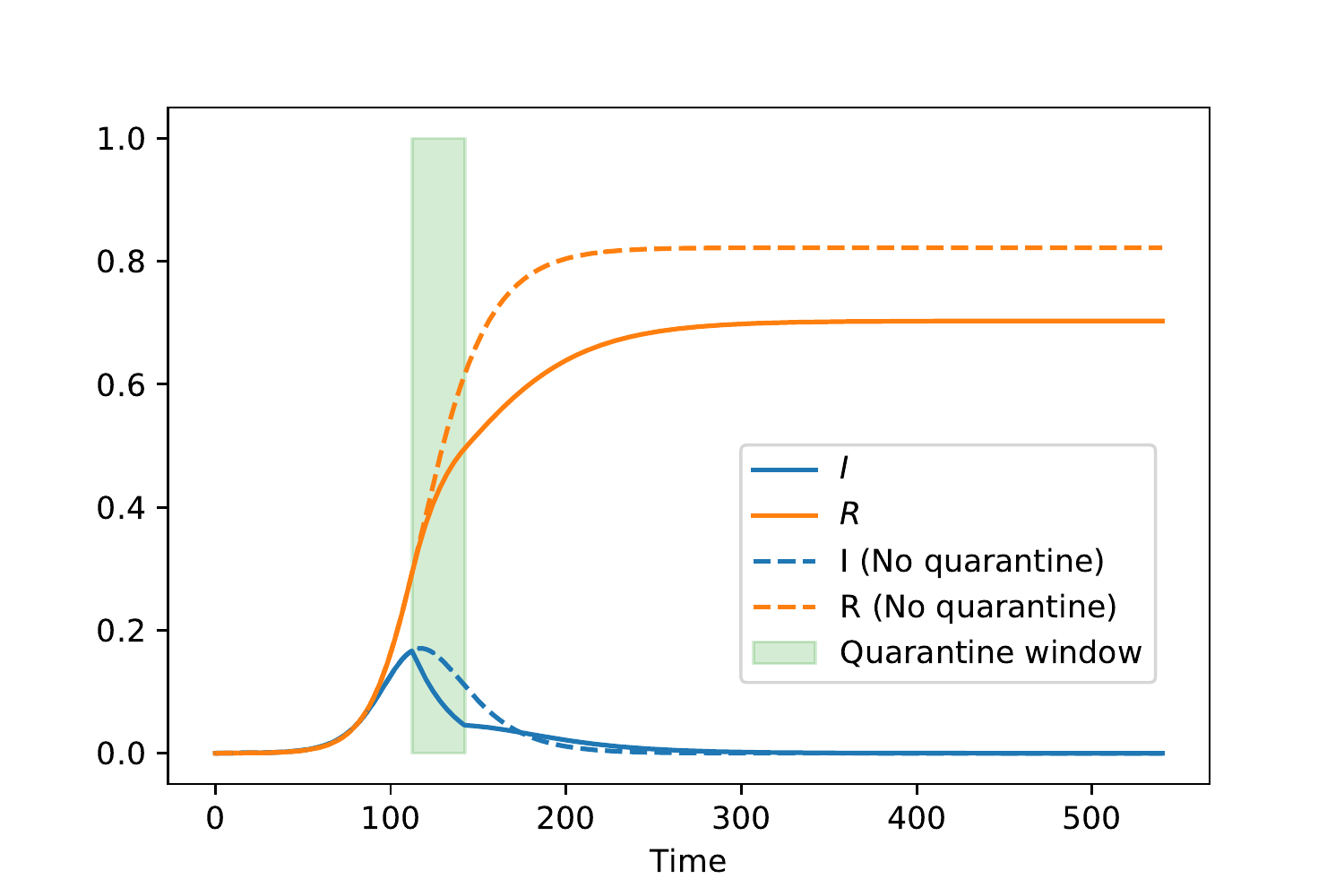}
  \quad
  \includegraphics[width=.48\linewidth]{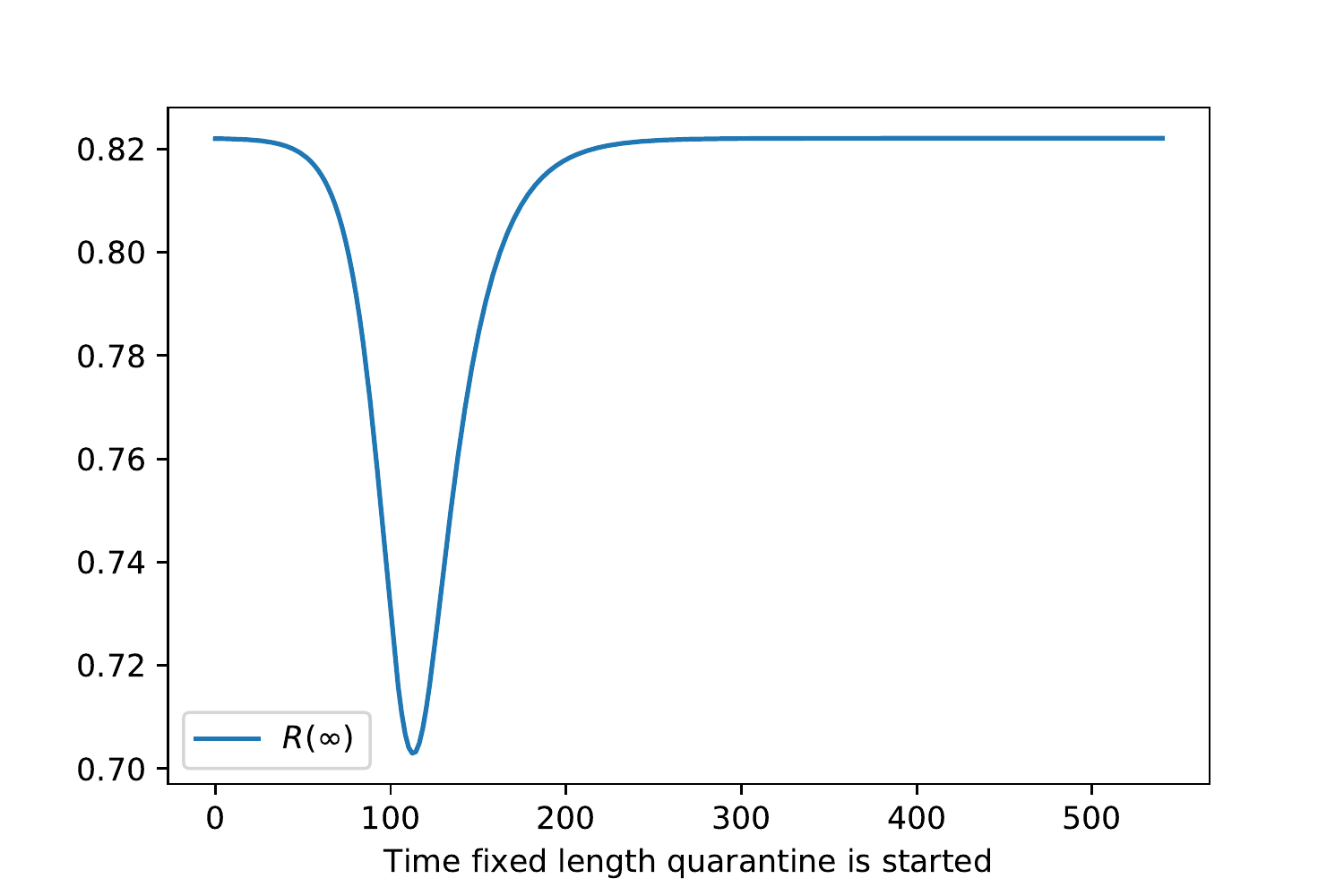}
  \caption{Left: $I$, $R$ vs $t$ both with a $30$ day, optimally scheduled, quarantine and without any quarantine.
  Right: The value of $R(\infty)$ vs the time when a $30$ day quarantine is started.}
  \label{f:IR}
\end{figure}

For comparison, we also plot how $R(\infty)$ varies based on the start of a $30$ day quarantine (Figure~\ref{f:IR}, right).
Here we see that when the quarantine is started too early, or too late, it has almost no impact on the value of $R(\infty)$.

Finally, in Figure~\ref{f:Rinf} we show how $R(\infty)$ varies when a~$30$ day quarantine is optimally imposed.
The two parameters we vary are $\mathcal R_0^n$, the basic reproduction number under normal circumstances, and~$\mathcal R_0^q$, the basic reproduction number when quarantines / social distancing are imposed.
Here we see that the reduction in $R(\infty)$ is larger when $\mathcal R_0^n$ is smaller.
\begin{figure}[hbt]
  \includegraphics[width=.48\linewidth]{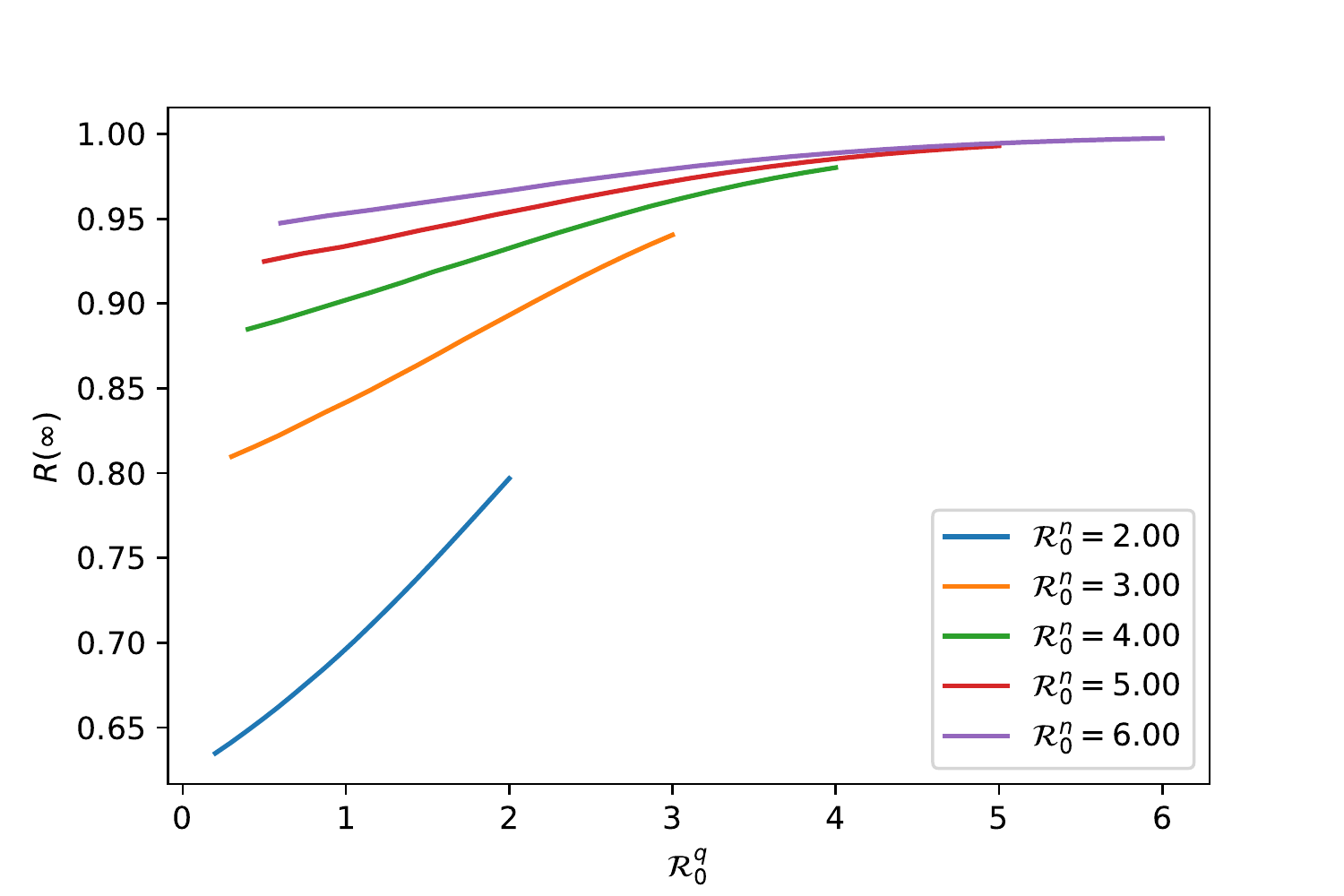}
  \quad
  \includegraphics[width=.48\linewidth]{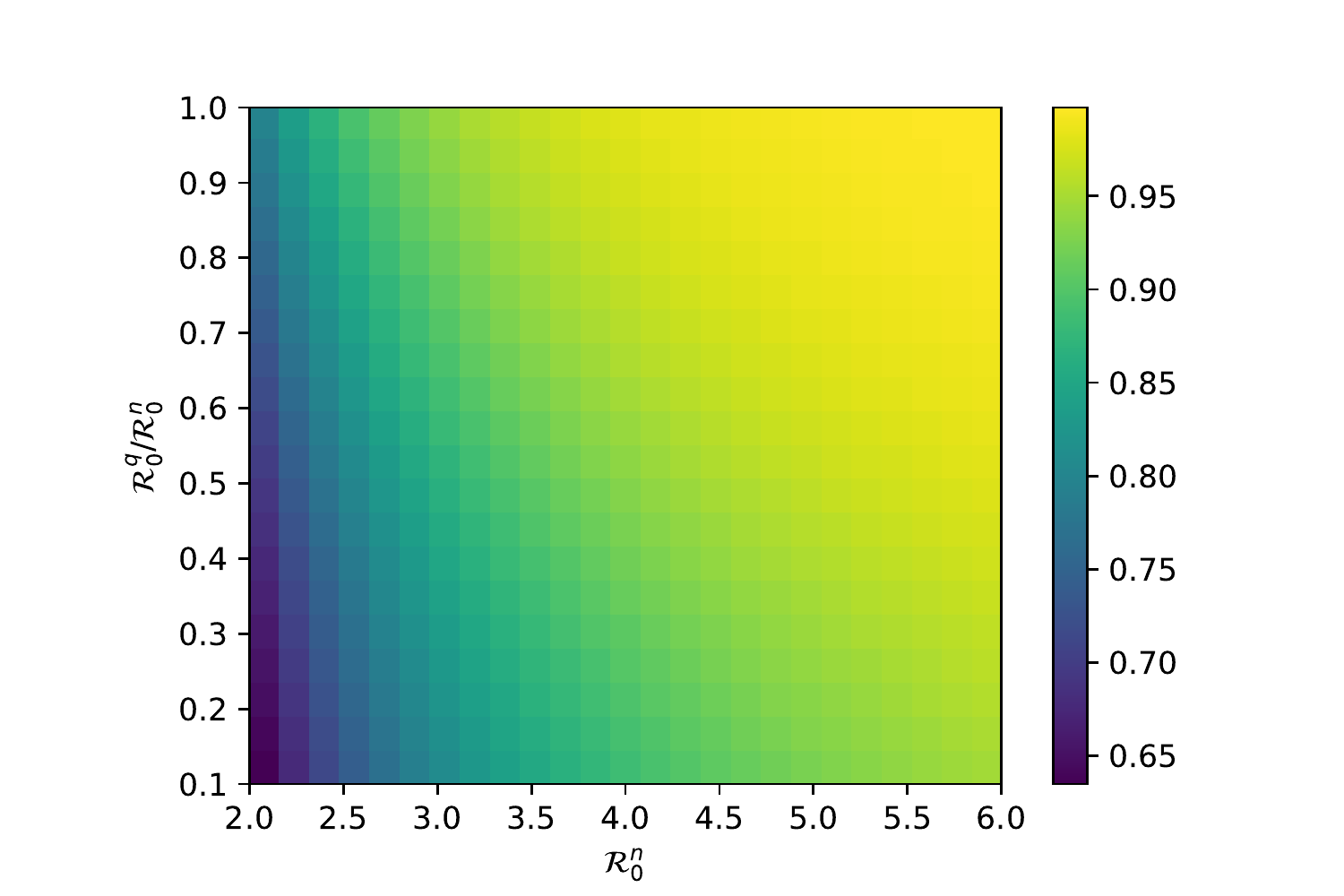}
  \caption{Minimum value of $R(\infty)$ when a $30$ day quarantine is optimally imposed.
    The figure on the left plots $R(\infty)$ vs $\mathcal R_0^q$ for a few different values of~$\mathcal R_0^n$.
    The figure on the right is a hot/cold plot of~$R(\infty)$ where $\mathcal R_0^n$ varies along the horizontal axis, and $\mathcal R_0^q/\mathcal R_0^n$  varies along the vertical axis.}
  \label{f:Rinf}
\end{figure}
\newpage

\bibliographystyle{halpha-abbrv}
\bibliography{refs,preprints}

\end{document}